\documentclass[12pt]{article}
\usepackage{amsmath,amsfonts,amssymb,amsthm,amscd}
\title{Strong dependence, weight, and measure}
\date{May 1st, 2010}
\author{Anand Pillay\thanks{Supported
by EPSRC grant EP/F009712/1, as well as an Invited Professorship at Universit\'e Paris-Sud 11 in March-April 2010.}\\University of Leeds}
\newtheorem{Theorem}{Theorem}[section]
\newtheorem{Proposition}[Theorem]{Proposition}
\newtheorem{Definition}[Theorem]{Definition} 
\newtheorem{Remark}[Theorem]{Remark}
\newtheorem{Lemma}[Theorem]{Lemma}
\newtheorem{Corollary}[Theorem]{Corollary}
\newtheorem{Fact}[Theorem]{Fact}

\newtheorem{Question}[Theorem]{Question}

\begin{document}
\maketitle

\begin{abstract} 
For a first order theory $T$ with $NIP$ I try to give an account of Shelah's notion of strong dependence, in terms
of suitable generically stable measures, forking, and ``weight".
\end{abstract}

\section{Introduction}
Shelah \cite{Shelah783} introduced the notion ``$T$ is strongly dependent" as an attempt to find an analogue of {\em superstability} in $NIP$ theories. When $T$ is stable, strong dependence is actually equivalent to ``all types have finite weight" (rather than superstability).  Here I give a version of this equivalence in the general $NIP$ context using suitable generically stable measures (see Theorem 1.1). There are two strong influences on this work. The first is the preprint \cite{OU} by Onshuus and Usvyatsov, which investigates strong dependence using an appropriate version of weight for types. The second is a talk by Hrushovski in Oberwolfach in January 2010 where he presented some tentative notions of ``finite weight" using generically stable measures. The results in the current paper are not hard and represent mainly a translation of certain notions into the language of measures. I also ask several questions, and hope to subsequently obtain a tighter account
of the relationship between the three notions in the title, depending on a better understanding of generically stable measures among other things. 
Thanks to Itai Ben Yaacov, Ehud Hrushovski, and Pierre Simon, for helpful discussions and communications.

In the remainder of this introduction, I will give an informal description of the basic notions, and then state the main result Theorem 1.1. In section 2 I will recall some material from \cite{NIPII}, \cite{NIPIII}, as well as giving precise definitions. In  section 3 I discuss {\em average measures} and prove some preliminary results. In section 4 I prove Theorem 1.1. I will assume a familiarity with stability theory, the ``stability-theoretic" approach to $NIP$ theories, as well as the notion of a Keisler measure. References are \cite{Pillay-book}, \cite{NIPI}, \cite{NIPII}, as well as papers of Shelah such as  \cite{Shelah863}. 
We will also be referring to Adler's paper \cite{Adler} which gives a nice treatment of the combinatorial notions around strong dependence, and makes explicit the connection with weight in the stable case.

Concerning notation, we work in a very saturated model ${\bar M}$ of a complete first order theory $T$ in 
language $L$. There is no harm to work in ${\bar M}^{eq}$, except that at some point we might want to make definitions concerning a given sort. $x,y,z,..$ denote finite tuples of variables. Likewise $a,b,c,..$ denote finite tuples of elements, and $M_{0}, M,..$ normally denote small elementary substructures of ${\bar M}$.

Recall that $T$ has $NIP$ (or is dependent) if for any indiscernible (over $\emptyset$) $(a_{i}:i<\omega)$, and formula $\phi(x,b)$, the truth value of $\phi(a_{i},b)$ is eventually constant. I will make a blanket assumption, at least in this introduction, that $T$ has $NIP$.

Our working definition of ``$T$  strongly dependent" (or ``strongly $NIP$") if there do NOT exist formulas $\phi_{\alpha}(x,\alpha)$, $k_{\alpha} < \omega$ and tuples $b^{\alpha}_{i}$, for $\alpha< \omega$, $i< \omega$, such that for each $\alpha$, $\{\phi_{\alpha}(x,b^{\alpha}_{i}):i<\omega\}$ is $k_{\alpha}$-inconsistent (every subset of size $k_{\alpha}$ is inconsistent), and for each $\eta\in \omega^{\omega}$, $\{\phi_{\alpha}(x,b^{\alpha}_{\eta(\alpha)}):\alpha < \omega\}$ is consistent. This is equivalent to Shelah's original definition assuming that $T$ has $NIP$. 

If $I = (a_{i}:i<\omega)$ is indiscernible over $\emptyset$, it is well-known that (assuming $NIP$) we can form the the average type $Avtp(I/A)$ over any set $A$, which is by definition the collection of formulas over $A$ eventually true of the $a_{i}$. More relevant to this paper is a construction in \cite{NIPIII} of the average {\em measure} of an indiscernible segment: If $I = (a_{i}:i\in [0,1])$ is an indiscernible {\em segment}, we can form the global average Keisler measure of $I$, $\mu_{I} = Avms(I/{\bar M})$, where for $\phi(x)$ a formula over ${\bar M}$, $\mu_{I}(\phi(x))$ is the Lebesgue measure of 
$\{i\in [0,1]: \models\phi(a_{i})\}$.  Such $\mu_{I}$ is a special case of a global {\em generically stable} measure, namely a global Keisler measure which is both finitely satisfiable in and definable over a small set (in this case $I$). By an average measure over $M_{0}$ we mean something of the form $\mu|M_{0}$ where $\mu$ is a global average measure which is $M_{0}$-invariant, or equivalently does not fork over $M_{0}$. More generally by a generically stable measure over $M_{0}$, we mean something of the form $\mu|M_{0}$ where $\mu$ is a (global) generically stable measure which is $M_{0}$-invariant. We note from \cite{NIPIII}  that any global generically stable measure $\mu$ which is $M_{0}$-invariant is the unique global nonforking (or $M_{0}$-invariant) extension of $\mu|M_{0}$. Namely any generically stable measure over $M_{0}$ has a {\em unique} global nonforking extension (where recall that a measure $\mu$ over $A$ does not fork over $M_{0}\subseteq A$ if any formula $\phi(x)$ with positive $\mu$-measure does not fork, or equivalently does not divide, over $M_{0}$). 

If $\mu_{x}$, $\lambda_{y}$ are Keisler measures over a small model $M_{0}$ and $\omega_{x,y}$ is a Keisler measure over $M_{0}$ extending 
$\mu \cup \lambda$, we can think of $\omega_{x,y}$ as an extension of $\mu_{x}$ over $\lambda_{y}$ and also as an extension of $\lambda_{y}$ over $\mu_{x}$. In section 2 we give a natural definition of the notion  ``$\omega_{x,y}$ is a {\em nonforking} extension of $\mu_{x}$ over $\lambda_{y}$".
In the case where  $\mu_{x}$ is generically stable, it will have a unique nonforking extension over any measure $\lambda_{y}$, which will coincide with the unique nonforking extension of $\lambda_{y}$ over $\mu_{x}$, and we write it as $\mu_{x}\otimes \lambda_{y}$ ($ = \lambda_{y}\otimes \mu_{x}$).

A key notion for the current paper is that of a {\em strong forking extension} of $\mu_{x}$ over $\lambda_{y}$. Namely an extension $\omega_{x,y}$ of $\mu$ over $\lambda$ is  a {\em strong forking extension} if there is a formula $\phi(x,y)$ over $M_{0}$ such that (i)  $\omega(\phi(x,y)) = 1$, and (ii) for ANY nonforking extension $\omega'(x,y)$ of $\mu_{x}$  (over some $\nu_{y})$), $\omega'(\phi(x,y)) = 0$.
So a strong forking extension of $\mu$ is a measure (in $x,y$ over $M_{0}$) which is 
``uniformly" orthogonal, in the sense of measure theory, to every nonforking extension $\omega'_{x,y}$ of $\mu_{x}$. We will mainly use this notion when $\mu_{x}$ is generically stable.

In a stable theory any forking extension of a strong type is a strong forking extension. 

Our main result is:
\begin{Theorem} Suppose $T$ has $NIP$. The the following are equivalent:
\newline
(1) $T$ is strongly dependent,
\newline
(2) It is NOT the case that there exist $M_{0}$ and average measures $\mu_{\alpha}(y_{\alpha})$ over $M_{0}$ for $\alpha < \omega$, and  measure $\omega_{x,y_{0},y_{1},.....}$ over $M_{0}$ such that
\newline
(i) $\omega_{x,y_{0},...} \supseteq  \otimes_{\alpha < \omega} \mu_{\alpha}(y_{\alpha})$,
\newline
(ii) for each $\alpha< \omega$, $\omega_{\alpha}(x,y_{\alpha})$ is a strong forking extension of $\mu_{\alpha}(x)$  (where $\omega_{\alpha}$ is just the restriction of $\omega$ to  $(x,y_{\alpha})$). 
\end{Theorem}

\begin{Remark}
(i)  Can we replace  ``average measures" in (2) by ``generically stable measures"? Of course with this change (2) implies (1) still holds but the problem is to prove (1) implies (2).
\newline
(ii) One can also give a version of Theorem 1.1 with average types $p_{\alpha}$ over $M_{0}$ in place of average measures (close to formulations in \cite{OU}), but with a less canonical notion of strong forking.
\end{Remark}

Finally in this introduction we discuss the meaning of Theorem 1.1 in terms of weight.
In a stable theory $T$, we say that a type $p(x)\in S(A)$ has {\em 
almost finite weight} if (possibly replacing $p$ by a nonforking extension) there do not exist an infinite $A$-independent 
set $\{b_{i}:i<\omega\}$ and realization $c$ of $p$ such that $c$ forks with $b_{i}$ over $A$ for all $i$. In fact, assuming 
stability, if $p$ has almost finite weight, then (again even ranging over nonforking extensions of $p$) there is a finite 
bound on the cardinality of any $A$-independent set $\{b_{i}:i\in I\}$ for which $c$ forks with $b_{i}$ over $M$ for all $i$. 
The greatest such cardinality is called the weight of $p$.  

Now in a stable theory and working with complete types over a model, types, average types and  average measures all coincide. So in (2) of Theorem 1.1, 
the $\mu_{\alpha}$ are already types. Replacing $\omega$ by a type and letting $(c,b_{\alpha})_{\alpha}$ be its realization, (2) says precisely that there do not exist a model $M_{0}$ and $c,b_{\alpha}$ for $\alpha < \omega$ such that $\{b_{\alpha}:\alpha < \omega\}$ is $M_{0}$-independent, and $tp(b_{\alpha}/M_{0},c)$ forks over $M_{0}$ for all $\alpha$. 

The fact that for $T$ stable, strong $NIP$ is equivalent to all types having finite weight, was pointed out already by Adler \cite{Adler}.

In any case our formulations in Theorem 1.1 suggest reasonable definitions of weight for types and measures in $NIP$ theories.

\section{Preliminaries}

The following definition is due to Shelah \cite{Shelah783}, and says that $\kappa_{ict} = \aleph_{0}$. 
\begin{Definition} $T$ is strongly dependent  (or strongly $NIP$) if there DO NOT exist formulas $\phi^{\alpha}(x,y^{\alpha})\in L$ for $\alpha < \omega$ and  
$(b^{\alpha}_{i})_{\alpha<\omega, i < \omega}$  such that for every 
$\eta\in \omega^{\omega}$, the set of formulas
$\{\phi^{\alpha}(x,b^{\alpha}_{\eta(\alpha)}): \alpha < \omega\} \cup 
\{\neg\phi^{\alpha}(x,b^{\alpha}_{i}:\alpha < \omega, i< \omega, i\neq \eta(\alpha)\}$ is consistent.
\end{Definition}  

\begin{Remark} (i) $T$ is strongly $NIP$ then $T$ is $NIP$.
\newline 
(ii) We can relativize the notion {\em strong $NIP$} to a sort $S$ by specifying that the variable $x$ in Definition 2.1 is of sort $S$.
\newline
(iii) In Definition 2.1 we could allow the $\phi^{\alpha}$ to have parameters  (by incorporating the parameters into the $b^{\alpha}$). 
\end{Remark}

\begin{Fact} Assume that $T$ has $NIP$. The following are equivalent (also sort by sort as far as the $x$ variable is concerned).
\newline
(1) $T$ is strongly $NIP$ in the sense of Definition 2.1.
\newline
(2) It is not the case that there exist formulas $\phi_{\alpha}(x,y_{\alpha})$ for $\alpha < \omega$, $b^{\alpha}_{i}$ for $\alpha < \omega$ and $i < \omega$, and $k_{\alpha} < \omega$ for each $\alpha < \omega$ such that
\newline
(i) for each $\alpha$, $\{\phi_{\alpha}(x,b^{\alpha}_{i}):i< \omega\}$ is $k_{\alpha}$-inconsistent, and
\newline
(ii) for each ``path" $\eta \in \omega^{\omega}$, $\{\phi_{\alpha}(x,b^{\alpha}_{\eta(\alpha)}):\alpha < \omega\}$ is consistent.
\newline
(3) Just like (2) but with a further clause
\newline
(iii) for each $\alpha$, the sequence $(b^{\alpha}_{i}:i<\omega)$ is indiscernible.
\end{Fact}
\begin{proof} This is contained in \cite{Adler}. See Propositions 10 and 13 there. But one can make further modifications of (2), for example demanding that each of the sequences $(b^{\alpha}_{i}:i<\omega)$ is indiscernible over the union of the other sequences. (See e.g. \cite{OU}.) Again one can allow parameters in the formulas.
\end{proof}

The following fact, due to Shelah \cite{Shelah863}, will be used. 
\begin{Fact} Suppose that $T$ is strongly dependent. Let $I$ be an (infinite) indiscernible ``sequence"  (i.e. indexed by some totally ordered set with respect to which it is indiscernible). Let $d$ be a finite tuple. Then we can write $I$ as the disjoint union of finitely many singletons and finitely many infinite``convex" subsets $I_{1},..,I_{k}$ such that each $I_{k}$ is indiscernible over $d$. 
\end{Fact}

In the papers \cite{NIPII}, \cite{NIPIII} we tended to focus very much on Keisler measures over ${\bar M}$ which do not fork over some small set. In the current paper we want to formulate our results in terms of measures over small sets (to be closer to the standard notions in stability theory). I will first briefly recall the relevant theory from the earlier papers, filling in a couple of gaps. I will assume that $T$ has $NIP$ from here on. 

A Keisler measure $\mu_{x}$ (sometimes also written $\mu(x)$) over $A$ is a finitely additive probability measure on the 
Boolean algebra of formulas $\phi(x)$ over $A$ up to equivalence (or of $A$-definable sets in sort $x$). Such $\mu$ 
can be identified with a regular Borel probability measure on the Stone space $S_{x}(A)$ of complete types over $A$ in 
variable $x$. By a global Keisler measure we mean a Keisler measure over $\bar M$.

\begin{Definition} Let $\mu_{x}$ be a Keisler measure over $B$, and let $A\subseteq B$. We say that $\mu$ does not fork over $A$ (or is a nonforking extension of $\mu|A$) if any formula $\phi(x)$ over $B$ with positive $\mu$-measure does not fork over $A$. 
\end{Definition}  

\begin{Remark} (i) It is easy to show, as in the case of types, that if $\mu$ is a Keisler measure over $B$ which does not fork over $A\subseteq B$ then $\mu$ has an extension over any $C\supseteq B$ (in particular over ${\bar M}$) which does not fork over $A$.
\newline
(ii) If $\mu_{x}$ is a Keisler measure over a model $M$, then $\mu$ does not fork over $M$, hence by (i) has a global nonforking extension. 
\end{Remark}

If $\mu_{x}$ is a global Keisler measure and $M_{0}$ is a small model, then the following are equivalent: (i) $\mu$ does not fork over $M_{0}$, (ii) $\mu$ is $Aut({\bar M}/M_{0})$-invariant, (iii) $\mu$ is Borel definable over $M_{0}$. 
The meaning of (iii) is that for any $L$-formula $\phi(x,y)$, and $b\in {\bar M}$, $\mu(\phi(x,b))$ depends in a Borel 
way on $tp(b/M_{0})$ in the sense that the function from $S_{y}(M_{0})$ to $[0,1]$ taking $tp(b/M_{0})$ to 
$\mu(\phi(x,b))$ is  Borel. A global measure $\mu_{x}$ satisfying (i) or (ii) or (iii) for some small $M_{0}$ is 
called {\em invariant}. So the point is that an invariant global Keisler measure $\mu_{x}$ automatically comes 
together with a Borel defining schema (so is a measure analogue of the notion of a parallellism class of stationary 
types in stability or what Hrushovski sometimes calls a ``definable type").  This additional information enables us to form a {\em canonical} ``nonforking" amalgam with any other global Keisler measure $\lambda_{y}$, which we called $\mu_{x}\otimes\lambda_{y}$.  Namely, given a formula $\phi(x,y,c)$ with parameters $c$ in ${\bar M}$, choose a small model $M_{0}$ containing the parameters such that $\mu$ is Borel definable over $M_{0}$, and one sees that the function taking $tp(b/M_{0})$ to $\mu(\phi(x,b,c))$ is a Borel function from $S_{y}(M_{0})$ to $[0,1]$ so can be integrated relative to the measure $\lambda|M_{0}$ on $S_{y}(M_{0})$, the value of which is defined to be $(\mu_{x}\otimes\lambda_{y})(\phi(x,y,c))$.

Among the other constructions associated to an invariant global Keisler measure $\mu_{x}$ are that of a ``Morley sequence"  $\mu^{(\omega)}_{x_{1},x_{2},...}$, a global Keisler measure in variables $x_{1},x_{2},...$  obtained by iterating the canonical nonforking amalgamations of $\mu_{x}$ with itself. It is ``indiscernible" (in obvious senses), and assuming that $\mu$ does not fork over $M_{0}$ also $\mu^{(\omega)}$ does not fork over $M_{0}$. More precisely $\mu^{(n)}$ is defined inductively  by  $\mu^{(n)}_{x_{1},..,x_{n}} =  \mu_{x_{n}}\otimes \mu_{x_{1},..,x_{n-1}}$ and $\mu^{(\omega)}_{x_{1},x_{2},..}$ is the union. 

An important class of invariant global Keisler measures is the class of {\em generically stable} measures, where (global) $\mu_{x}$ is said to be generically stable if it is both finitely satisfiable in and definable over some small model $M_{0}$. Definability here is the strengthening of Borel definability to:  $\mu(\phi(x,b))$ depends in a continuous way on $tp(b/M_{0})$. Among key properties of a global generically stable measure $\mu_{x}$ is that whenever $\mu$ does not fork over small model $M$ then $\mu$ is actually the {\em unique} global nonforking extension of $\mu|M$. Also that for any other invariant global Keisler measure, $\mu_{x}\otimes \lambda_{y} = \lambda_{y}\otimes \mu_{x}$. 

Now we pass to measures $\mu_{x}$ over small sets $A$, mainly models $M_{0}$. Given complete types $p(x)$,  $q(y)$, and $r(x,y)$ over $A$ such that $r$ extends $p\cup q$ we can think of $r$ as an extension of $p$ over $q$ 
or an extension of $q$ over $p$. Considering $r$ as an extension of $p(x)$ over $q(y)$ it is natural to call $r$ a ``nonforking" extension if for some (any) $b$ realizing $q$, $r(x,b)$ does not fork over $A$. Of course because we are used to realizing types, we rather speak of $r(x,b)$ as a nonforking extension of $p(x)$ over $A,b$. 

Now, unless we are very familiar with continuous model theory, we do not really want to talk about realizing measures, but we do want the notion of a nonforking extension of a measure $\mu_{x}$ over another measure $\lambda_{y}$  (both measures over $A$). I will give below a rather forced definition, making use of global measures. And I will leave it to subsequent work to show the agreement of our definition with other more natural ones, coming from continuous logic for example. We will work over (small) models, so as to have the existence of nonforking extensions. 

If $\mu'_{x}$ is a global Keisler measure which does not fork over $M_{0}$, and $\phi(x,y)$ a formula over $M_{0}$, we saw above that the function taking $tp(b/M_{0})$ to $\mu'(\phi(x,b))$ is Borel from $S_{y}(M_{0})$ to $[0,1]$.  We call this function $f_{\mu',\phi}$. 

\begin{Definition} Let $\omega_{x,y}$, $\mu_{x}$, $\lambda_{y}$  be Keisler measures over $M_{0}$ such that $\omega_{x,y}$ extends  $\mu_{x}\cup\lambda_{y}$. We will say that  $\omega$ is a nonforking extension of $\mu_{x}$ 
over $\lambda_{y}$ if  for some global nonforking extension $\mu'$ of $\mu$, for every formula $\phi(x,y)$ over $M_{0}$, $\omega(\phi(x,y)) = \int_{q\in S_{y}(M_{0})}f_{\mu',\phi}(q) d\lambda$. 
\end{Definition}

\begin{Remark} (i) So, $\omega_{x,y}$ is a nonforking extension of $\mu_{x}$ over $\lambda_{y}$ if and only for some global nonforking extension $\mu'$ of $\mu$ and some (any) global extension $\lambda'$ of $\lambda$, $\omega_{x,y} = (\mu'_{x}\otimes \lambda'_{y})|M_{0}$.
\newline
(ii) By Remark 2.6(ii), for any Keisler measures $\mu_{x}$, $\lambda_{y}$ over $M_{0}$, there exists some nonforking extension of $\mu$ over $\lambda$.
\end{Remark}

\begin{Lemma} Let $\mu_{x}$ be a measure over $M_{0}$, and $\phi(x,y)$ a formula over $M_{0}$. Then the following are equivalent:
\newline
(i) for every measure $\omega_{x,y}$ over $M_{0}$ which is a nonforking extension of $\mu_{x}$ (over the restriction of $\omega$ to $y$), $\omega_{x,y}(\phi(x,y)) = 0$.
\newline
(ii) For every global nonforking extension $\mu'_{x}$ of $\mu$ over ${\bar M}$, and $b\in {\bar M}$ $\mu'(\phi(x,b)) = 0$.
\end{Lemma}
\begin{proof}  The way our definitions are set up this is immediate. For example if (ii) fails then $\mu$ has a nonforking extension $\omega(x,b)$ over $b$ with $\omega(\phi(x,b)) > 0$, and  so $\omega(x,y)$ is a nonforking extension of $\mu_{x}$ over $tp(b/M_{0})$ such that $\omega(\phi(x,y)) > 0$. 
\end{proof}

\begin{Definition}  Let again $\omega_{x,y}$ over $M_{0}$ be an extension of $\mu_{x}$ over $\lambda_{y}$. We will say that  $\omega$ is a strong forking extension  (of $\mu$ over $\lambda$) if for some $L_{M_{0}}$-formula $\phi(x,y)$, $\omega_{x,y}(\phi(x,y)) = 1$  but for every nonforking extension $\omega'_{x,y}$ of $\mu_{x}$ over any $\nu_{y}$, 
$\omega'(\phi(x,y)) = 0$.
\end{Definition}

\begin{Remark} Suppose that $T$ is stable. If $tp(c/M_{0}b)$ forks over $M_{0}$ then it strongly forks over $M_{0}$, in the sense that for some formula $\phi(x,y)$ over $M_{0}$, $\models \phi(c,b)$ but  for every $b'\in {\bar M}$
$\neg\phi(x,b')$ is in every global nonforking extension $p$ of $tp(c/M_{0})$.
\end{Remark}
\begin{proof} By uniqueness and definability of $p$.
\end{proof}

\begin{Lemma}  Suppose $\mu_{x}$ is a generically stable measure over $M_{0}$, and $\lambda_{y}$ is any other measure over $M_{0}$. Then $\mu_{x}$ has a unique nonforking extension $\omega_{x,y}$ over $\lambda_{y}$, which is also the
unique nonforking extension of $\lambda_{y}$ over $\mu_{x}$. We write this unique extension as $\mu_{x}\otimes \lambda_{y}$ or $\lambda_{y} \otimes \mu_{x}$. 
\end{Lemma}
\begin{proof} This is immediate from the definitions together with uniqueness of global nonforking extension of genercially stable measure  plus the fact that a global generically stable measure ``commutes" with any invariant global measure.
\end{proof}

\begin{Remark}  So there is some ambiguity with our notation $\mu_{x}\otimes \lambda_{y}$, depending on whether $\mu, \lambda$ are global types or not.  But if  $\mu_{x}$ over $M_{0}$ is generically stable, $\mu'$ is its unique global nonforking extension, and $\lambda_{y}$ any measure over $M_{0}$, then $\mu_{x}\otimes \lambda_{y}$ agrees with the restriction to $M_{0}$ of  $\mu'_{x}\otimes \lambda'_{x}$  for any  global extension $\lambda'$ of $\lambda$.
\end{Remark}

\begin{Question} What can one say about the symmetry of strong forking, for example when one or both of $\mu_{x}$, $\lambda_{y}$ are generically stable?
\end{Question}

Finally in this section we recall from \cite{NIPIII}:
\begin{Definition} Let $\mu_{x}$ be a Keisler measure over $A$  (where here we allow $x$ to be infinite tuple of variables). Then a tuple $c$ (of same length as $x$) is said to be weakly random over $A$ for $\mu$, if
$\models \neg\phi(c)$, for all formulas $\phi(x)$ over $A$ with $\mu$-measure $0$.
\end{Definition}

Note that, assuming $A$ to be small, such weakly random elements exist, by compactness.

\section{Average measures}

We now discuss a very special class of generically stable measures, namely average measures. 
\begin{Definition}  By an indiscernible segment we mean something of the form $\{a_{i}:i\in [0,1]\}$ which is indiscernible with respect to the usual ordering on $[0,1]$.
\end{Definition}

As pointed out in \cite{NIPIII} such an indiscernible segment $I$ gives rose to a global generically stable measure
$\mu_{I}$:
for any formula (with parameters) $\phi(x)$ the set of $i\in [0,1]$ such that $\models\phi(a_{i})$ is a finite union of intervals and points so has a Lebesgue measure, which we define to be $\mu_{I}(\phi(x))$. Noting that $\mu_{I}$ is both finitely satisfiable in and definable over $I$, we see that $\mu_{I}$ is a global generically stable measure. 

\begin{Definition} (i) By a global average measure we mean something of the for $\mu_{I}$ for $I$ an indiscernible segment.
\newline
(ii) For $M_{0}$ a small model, by a average measure over $M_{0}$ we mean something of the form $\mu_{I}|M_{0}$ where $\mu_{I}$ is a global average measure which does not fork over $M_{0}$ (or is $Aut({\bar M}/M_{0})$-invariant). 
\end{Definition}

\begin{Remark} A generically stable {\em type}  is the same thing as an average measure which happens to be a type.
\end{Remark}

\begin{Lemma}  Suppose $T$ is strongly dependent. Let $I = (a_{i}:i\in [0,1])$ be an indiscernible segment, and suppose that $\phi(x,y)$ is a formula over ${\bar M}$ such that $\mu_{I}(\phi(x,b)) = 0$  for all $b\in {\bar M}$. Then for some $k$, $\{\phi(a_{i},y):i\in [0,1]\}$ is $k$-inconsistent.
\end{Lemma}
\begin{proof}  Let $d$ be the parameters from $\phi(x,y)$. Let $I_{1},..,I_{m}$ be given by Fact 2.4. For each $j$,
$\{\phi(a_{i},y):a_{i}\in I_{j}\}$ is clearly inconsistent, as otherwise if it is realized by $c$, then $\mu_{I}(\phi(x,c))$ is at least the length of $I_{j}$. By indiscernibility of $I_{j}$ over the parameter $d$, it follows that for some $k_{j}$, $\{\phi(a_{i},y):a_{i}\in I_{j}\}$ is $k_{j}$-inconsistent. So for suitable $k$  (depending also on the cardinality of the finite set $I\setminus \{I_{1}\cup .. \cup I_{m}\}$), $\{\phi(a_{i},y):a_{i} \in I\}$ is $k$-inconsistent.

\end{proof}

\begin{Remark}  If $\phi(x,y)$ has no parameters then there is of course no need for the ``strong $NIP$" assumption in 3.4, and we have in fact that $\{\phi(a_{i},y):a_{i}\in I\}$ is inconsistent iff $\mu_{I}(\phi(x,b)) = 0$ for all $b$ iff  $\mu_{I}\phi(x,b)) < 1$ for all $b$.
\end{Remark}

\begin{Lemma} Suppose $I$ is an indiscernible segment. Let $\phi(x_{1},..,x_{n})$ be a formula over ${\bar M}$. 
Suppose $\mu_{I}^{(n)}(\phi(x_{1},...,x_{n})) > 0$. Then 
\newline
$\exists^{\infty}x_{n}\in I(\exists^{\infty}x_{n-1}\in I(...(\exists^{\infty}x_{1}\in I (\models \phi(x_{1},..,x_{n}))...)$.
\end{Lemma}
\begin{proof}  The converse is also true but I will only prove (by induction) the stated direction. The case $n=1$ is 
immediate by definition of $\mu_{I}$. Suppose true for $n$. Suppose $\mu_{I}(\phi(x_{1},..,x_{n+1})) > 0$. So by definition of 
the ``nonforking product" measure, for some $\epsilon \geq 0$, $Y = \{(c_{1},..,c_{n}): \mu(\phi(c_{1},..,c_{n},x_{n+1}) \geq \epsilon\}$ is 
type-definable over $I$ and the parameters from $\phi$, and) has positive $\mu^{(n)}$-measure. Let ${\bar c}\in Y$ be chosen
weakly random for the restriction of $\mu^{(n)}$ to $I$ together with the parameters in $\phi$. As $\mu(\phi({\bar c},x_{n+1})) > 0$ it is realized by infinitely
many elements of $I$ (in fact an ``interval"), say $Z$. For each $d\in Z$, $\mu^{(n)}(\phi(x_{1},..,x_{n},d)) > 0$  (as the formula holds of the weakly random ${\bar c}$) so we can apply the induction hypothesis.
\end{proof}

\begin{Corollary} Assume $T$ to be strongly dependent. Suppose that $I$ is an indiscernible segment. Let $\phi(x,y)$ be a formula over ${\bar M}$ such that $\mu_{I}(\phi(x,b)) = 0$ for all $b\in {\bar M}$. Then for some $k$, $\mu_{I}^{(k)}(\exists y(\phi(x_{1},y) \wedge ... \wedge\phi(x_{k},y))) = 0$.
\end{Corollary}
\begin{proof}  Let $k$ be given by Lemma 3.4. Let $\theta(x_{1},..,x_{k})$ be the formula  
$(\exists y)(\phi(x_{1},y) \wedge ... \wedge\phi(x_{k},y))) = 0$. If $\mu_{I}^{(k)}(\theta) > 0$ then by Lemma 3.6  we find distinct  $a_{i_1},..,a_{i_k}\in I$ such that $\models \theta(a_{i_1},..,a_{i_k})$, contradicting the choice of $k$.
\end{proof}

Finally in this section we give prove a general proposition crucial for the main theorem.
\begin{Proposition} Suppose that $\mu_{1}(y_{1}),...,.\mu_{n}(y_{n})$ are global Keisler measures, all invariant over a small model $M_{0}$. Let $\mu(y_{1},..,y_{n})$ be the nonforking product $\mu_{1}\otimes .. \otimes \mu_{n}$. Let $B(y_{1},..,y_{n})$ be a Borel set over $M_{0}$ with $\mu$-measure $1$. Then there are sequences  $I_{\alpha} = (b^{\alpha}_{i}:i<\omega)$ for $\alpha = 1,..,n$ such that
\newline
(i) each $I_{\alpha}$ is weakly random for $(\mu_{\alpha})^{(\omega)}|M_{0}$
\newline
(ii) for all  $(c_{1},..,c_{n}) \in I_{1} \times .. \times I_{n}$, $(c_{1},..,c_{n})\in B$.
\end{Proposition}
\begin{proof}  We argue by induction on $n$.  For $n=1$, let $x$ be the variable $y_{1}$. Then
the intersection of all the $B(x_{i})$ for $i< \omega$ together with all $M_{0}$-definable set of $\mu_{1}^{(\omega)}$-measure 
$1$ is a Borel subset of the type space over $M_{0}$ in variables $(x_{1},x_{2},..)$ of $\mu_{1}^{(\omega)}$-measure $1$, 
hence contains a point, and any realization is the required $I_{1}$.
\newline
Assume true for $n$. Let $B(y_{1},..,y_{n+1})$ be a Borel set over $M_{0}$ of $\mu$ measure $1$, where $\mu = 
\mu_{1}\otimes... \otimes \mu_{n+1}$). By Borel definability of invariant measures, and the definition of the nonforking product measure, $\{(c_{2},...,c_{n+1}): \mu_{1}(B(y_{1},c_{2},..,c_{n+1})) = 1\}$ is a Borel set $C(y_{2},..,y_{n+1})$ over $M_{0}$ of $(\mu_{2}\otimes ... \otimes \mu_{n+1})$-measure $1$. 
By induction hypothesis we find $I_{2},..,I_{n+1}$ satisfying (i) and (ii) of the Proposition for $C$ in place of $B$. 
Now again let $x$ be the variable $y_{1}$. Consider the countable set of conditions  $B(x_{i},c_{2},..,c_{n+1})$ for $i<\omega$ and  $(c_{2},..,c_{n+1})\in I_{2}\times .. \times I_{n+1}$. The intersection of all of these is a Borel set in variables $(x_{1},x_{2},...)$ which has $\mu_{1}^{(\omega)}$-measure $1$. The intersection of  this 
with the set of all formulas over $M_{0}$ of $\mu_{1}^{(\omega)}$-measure $1$, again has a point, which is the required
$I_{1}$.

\end{proof}

\section{Proof of Theorem 1.1}
We start with
\newline
{\bf Proof of (2) implies (1).}
\newline
Assume that (1) fails. By Fact 2.3, there are $\phi_{\alpha}(x,y_{\alpha})\in L$ for $\alpha < \omega$, $b^{\alpha}_{i}$ for $\alpha < \omega$ and $i < \omega$, and $k_{\alpha} < \omega$ for each $\alpha < \omega$ such that
\newline
(i) for each $\alpha$, $\{\phi_{\alpha}(x,b^{\alpha}_{i}):i< \omega\}$ is $k_{\alpha}$-inconsistent, and
\newline
(ii) for each ``path" $\eta \in \omega^{\omega}$, $\{\phi_{\alpha}(x,b^{\alpha}_{\eta(\alpha)}):\alpha < \omega\}$, and
\newline
(iii) for each $\alpha$, the sequence $(b^{\alpha}_{i}:i< \omega)$ is indiscernible.

\vspace{2mm}
\noindent
By compactness we may find $b^{\alpha}_{i}$ for $\alpha < \omega$ and $i\in [0,1]$ satisfying the analogues of (i), (ii), (iii).
(So in (i) we now have $\eta \in [0,1]^{\omega}$, and in (iii) we have an indiscernible segment.)

Let $I_{\alpha}$ denote the indiscernible segment $(b^{\alpha}_{i}:i\in [0,1])$.
\newline
{\em Claim 1.} For each $\alpha$, for every $b$, $\mu_{I_{\alpha}}(\phi_{\alpha}(b,y_{\alpha})) = 0$.
\begin{proof}  This is immediate from the assumption that $\{\phi(x,b^{\alpha}_{i}):i\in [0,1]\}$ is ($k_{\alpha}$-) inconsistent.
\end{proof}

\vspace{2mm}
\noindent
{\em Claim 2.} For each $n < \omega$, $(\otimes_{\alpha < n}\mu_{I_{\alpha}})(\exists x(\wedge_{\alpha < n}\phi_{\alpha}(x,y_{\alpha}))) = 1$.
\begin{proof} This follows  easily from the definitions (including of the nonforking product) and our assumption (ii) above, that for ANY $i_{0},i_{1},..,i_{n-1}\in [0,1]$ $\{\phi_{\alpha}(x,b^{\alpha}_{i_{\alpha}}):\alpha < n\}$ is consistent.

\end{proof}

Let $M_{0}$ be a small model such that all the $\mu_{I_{\alpha}}$ are $M_{0}$-invariant. Let $\mu_{\alpha}(y_{\alpha})$
 be the restriction of $\mu_{I_{\alpha}}$ to $M_{0}$. So the $\mu_{\alpha}$ are average measures over $M_{0}$ and 
$\otimes_{\alpha < \omega}\mu_{\alpha}$ is the restriction of $\otimes_{\alpha<\omega}\mu_{I_{\alpha}}$ to $M_{0}$.  By 
Claim 2, $(\otimes_{\alpha < n}\mu_{\alpha})(\exists x(\wedge_{\alpha < n}\phi_{\alpha}(x,y_{\alpha}))) = 1$ for all $n$.
Hence there is a measure $\omega_{x,y_{0},y_{1},...}$  such that $\omega(\phi_{\alpha}(x,y_{\alpha})) = 1$ for all $\alpha$, and thus if $\omega_{\alpha}$ denotes the restriction of $\omega$ to $(x,y_{\alpha})$, $\omega_{\alpha}(\phi_{\alpha}(x,y_{\alpha})) = 1$ for all $\alpha$. By Claim 1, and Lemma 2.9, $\omega_{\alpha}(x,y_{\alpha})$ is a strong forking extension of $\mu_{\alpha}(y_{\alpha})$.

\noindent
Claims 1 and 2  show that the $\mu_{\alpha}(y_{\alpha})$ and $\phi_{\alpha}(x,y_{\alpha})$ witness the failure of (2). So we have proved (2) implies (1). 

\vspace{5mm}
\noindent
{\bf Proof of (1) implies (2).}
\newline
Let us assume that $T$ is strongly dependent, and that (2) fails, and aim for a contradiction. The assumption that (2) fails gives average measures $\mu_{\alpha}(y_{\alpha})$ over $M_{0}$ for $\alpha < \omega$ and $\omega_{x,y_{0},y_{1},...}$ over $M_{0}$ extending $\otimes_{\alpha}\mu_{\alpha}$ such that the restriction $\omega_{\alpha}$ to $(x,y_{\alpha})$ is a strong forking extension of $\mu_{\alpha}(y_{\alpha})$ for all $\alpha$.

\noindent
Let $\mu_{\alpha} = \mu_{I_{\alpha}}|M_{0}$ for some  indiscernible segment  $I_{\alpha} = (b_{i}^{\alpha}:i\in [0,1])$.

\noindent
For each $\alpha < \omega$, let $\phi_{\alpha}(x,y_{\alpha})$ be a formula over $M_{0}$ (so possibly with parameters) witnessing the strong forking. Namely $\omega_{\alpha}(\phi_{\alpha}(x,y_{\alpha})) = 1$ but $\omega'(\phi_{\alpha}(x,y_{\alpha})) = 0$ for any nonforking extension $\omega'_{x,y_{\alpha}}$ of $\mu_{\alpha}(y_{\alpha})$. By Lemma 2.9 as well as the fact that $\mu_{I_{\alpha}}$ is the unique global nonforking extension of $\mu_{\alpha}$, it follows that $\mu_{I_{\alpha}}(\phi_{\alpha}(c,y_{\alpha})) = 0$ for all $c\in {\bar M}$, and hence by Lemma 3.4, for each $\alpha$ there is $k_{\alpha}$ such that  $\{\phi_{\alpha}(x,b_{i}^{\alpha}):i\in [0,1]\}$ is $k_{\alpha}$-inconsistent. 

Note that for each $n$, 
\newline
(*) $\otimes_{\alpha < n} \mu_{\alpha}(\exists x(\wedge_{\alpha < n} \phi_{\alpha}(x,y_{\alpha}))) = 1$. 
\newline
We will use this data to construct for each $n$ some $\{d^{\alpha}_{i}:\alpha < n, i< \omega\}$ such that  
$\{\phi_{\alpha}(x,d^{\alpha}_{i}):i<\omega\}$ is $k_{\alpha}$-inconsistent for each $\alpha < n$, and each ``path", 
$\phi_{\alpha}(x,d^{\alpha}_{\eta(\alpha)}):\alpha < n\}$ is consistent. Then by compactness we will obtain a witness to non strong dependence (in the form given in Fact 2.3, and allowing parameters in the $\phi_{\alpha}(x,y_{\alpha})$).  

So fix $n$. Now (*) and Proposition 3.8  give sequences $I_{\alpha}$  for $\alpha < n$  such that  $I_{\alpha}$ is weakly random for $\mu_{\alpha}^{(\omega)}$  (the latter being the same as
$\mu_{I_{\alpha}}^{(\omega)}|M_{0}$), and such that 
$\models \exists x(\phi_{0}(x,c_{0}) \wedge \phi_{1}(x,c_{1}) \wedge .. \wedge \phi_{n-1}(x,c_{n-1}))$  for all $(c_{0},..,c_{n-1}) \in I_{0}\times .. \times I_{n-1}$.  Weak randomness of $I_{\alpha}$ together with 
Corollary 3.7  imply that $\{\phi_{\alpha}(x,c): c\in I_{\alpha}\}$ is $k_{\alpha}$-inconsistent. So we have found $\{d^{\alpha}_{i}:\alpha < n, i < \omega\}$ as required.  
\newline
The proof of Theorem 1.1 is complete.

\end{document}